\documentclass[12pt]{article}

\usepackage{amsmath}
\usepackage{amsthm}
\usepackage{amsfonts}
\usepackage{mathrsfs}
\usepackage{stmaryrd}
\usepackage{setspace}
\usepackage{fullpage}
\usepackage{amssymb}
\usepackage{breqn}
\usepackage{enumitem}
\usepackage{bbold}
\usepackage{authblk}
\usepackage{comment}
\usepackage{hyperref}
\usepackage{pgf,tikz}
\usepackage{graphicx}
\usepackage{subcaption}

\newtheorem{thm}{Theorem}[section]

\newtheorem{lemma}[thm]{Lemma}

\newtheorem{proposition}[thm]{Proposition}
\newtheorem{prop}[thm]{Proposition}

\newtheorem{clm}[thm]{Claim}

\newcommand\ex{\ensuremath{\mathrm{ex}}}

\newcommand\cF{{\mathcal F}}
\newcommand\cG{{\mathcal G}}
\newcommand\cH{{\mathcal H}}

\newcommand\cW{{\mathcal W}}

\newtheorem*{thm*}{Theorem}
\newtheorem*{prop*}{Proposition}
\newcommand{\ignore}[1]{}

\title{On the Tur\'{a}n number of the expansion of the book}
\author{
Xin Cheng\thanks{\small School of Mathematics and Statistics, Northwestern Polytechnical University and Xi'an-Budapest Joint Research Center for Combinatorics, Xi'an 710129, Shaanxi, P.R. China. Email:
\small \texttt{xincheng@mail.nwpu.edu.cn}.}\,, \hspace{0.2em}
D\'{a}niel Gerbner\thanks{\small Alfr\'ed R\'enyi Institute of Mathematics. Email:
\small \texttt{gerbner.daniel@renyi.hu}.}\,, \hspace{0.2em} 
Hilal Hama Karim$^\dagger$\thanks{\small Department of Computer Science and Information Theory, Faculty of Electrical Engineering and Informatics, Budapest University of Technology and Economics, Műegyetem rkp. 3., H-1111 Budapest, Hungary. E-mail: \texttt{hilal.hamakarim@edu.bme.hu}.}}

\date{}

\begin{document}

\maketitle

\begin{abstract}

 The book with $t$ pages is the graph on $t+2$ vertices consisting of $t$ triangles which intersect at exactly one common edge. For a given graph $F$, the $r$-expansion $F^r$ of $F$ is the $r$-uniform hypergraph obtained from $F$ by adding $r-2$ distinct new vertices to each edge of $F$. We determine the Tur\'an number of the 3-expansion of the book graph for sufficiently large $n$.

\end{abstract}



\section{Introduction}


Given an $r$-uniform hypergraph $\cF$, its \textit{Tur\'an number} $\ex_r(n,\cF)$ is the maximum number of hyperedges in an $n$-vertex $r$-uniform hypergraph that does not contain $\cF$ as a subgraph. Tur\'an numbers are among the must studied and yet not well-understood problems in extremal Combinatorics, see \cite{keevash} for a survey. Much more is known in the case $r=2$, i.e., graph Tur\'an problems, see \cite{fursim} for a survey.

This motivates a recent trend where we define graph-based hypergraphs by enlarging graph edges, in the hope of applying graph-theoretic methods. In particular, the $r$-uniform \textit{expansion} $F^r$ of a graph $F$ is obtained from $F$ by adding $r-2$ distinct new vertices to each edge of $F$ (thus it has $(r-2)|E(F)|+|V(F)|$ vertices and $|E(F)|$ edges). Expansions were introduced by Mubayi \cite{Mu} for cliques, see \cite{MuVe} for a survey.

Some results of importance for us is $\ex_r(n,K_3)=\binom{n-1}{r-1}$ \cite{fra, MuVe2} and $\ex_3(n,F_t)=\binom{n}{3}-t\binom{n}{2}$ \cite{tfan} for sufficiently large $n$, where the $t$-fan $F_t$ consists of $t$ triangles sharing a vertex. Here we continue this line of research by determining the Tur\'an number of $B_t^3$ for sufficiently large $n$, where the book graph $B_t$ consists of $t$ triangles sharing an edge. Let us remark that $\ex_2(n,B_t)=\lfloor n^2/4\rfloor$ for sufficiently large $n$ by a theorem of Simonovits \cite{sim}.

\begin{thm}\label{thmnew1}
For any $t>1$, if $n$ is sufficiently large, then
  $\ex_3(n,B_t^3)=t\binom{n-t}{2}$. 
\end{thm}

In Section \ref{sec2}, we collect some lemmas. The proof of Theorem \ref{thmnew1} is in Section \ref{sec3}, and Section \ref{sec4} contains some concluding remarks.

\section{Preliminaries}\label{sec2}

For a book graph $B_s$ with vertex set $\{u,v, x_1, \ldots, x_s\}$ and edge set $\{uv, ux_i, v_i : i=1, \ldots, s\}$, we call the edge $uv$ the \textit{rootlet}, the edges $ux_i$ and $vx_i$ the \textit{page edges}, and the vertices $x_i$ the \textit{page vertices} of the book. A \textit{red-blue graph} is a graph where the edges have color red or blue or both. The \textit{red degree} $d_r(v)$ of a vertex $v$ is the number of red edges (including the edges that have both colors) incident to $v$, and the blue degree $d_b(v)$ is defined analogously. For simplicity we say $v$ is a red(respectively, blue) neighbor of $u$ if the edge $uv$ has color red (respectively, blue).  Our goal is to find \textit{rainbow cherries}, i.e., adjacent blue and red edges.

\begin{lemma}\label{lem}
    Let $G$ be a red-blue graph that contains a set $U$ of at least $4t+1$ vertices such that $d_r(u)\ge t$ and $d_b(u)\ge 2t$ for each $u\in U$. Then $G$ contains $t$ vertex-disjoint rainbow cherries.
\end{lemma}

\begin{proof} Assume that the statement does not hold.
    Choose a largest family $\mathcal{W}$ of vertex-disjoint rainbow cherries and let $W$ denote the set of their vertices. Then $|W|=3k\le 3t-3$. Let $U'=U\setminus W$, thus $|U'|\ge t+4$. For each $u\in U'$, since we cannot extend our family of rainbow cherries with a rainbow cherry centered at $u$, either all the red neighbors of $u$ are in $W$, or all the blue neighbors of $u$ are in $W$, or $u$ has exactly one red neighbor and one blue neighbor outside $W$, and they are the same vertex $v$.

    Assume that there are two red edges $uv$, $u'v'$, such that $u,u'\in U'$, $u\neq u'$ and $v,v'\not\in W$. Then $u,u'$ have at least $2t-1$ blue neighbors in $W$. In particular, there is a rainbow cherry $v_1v_2v_3$ in $\mathcal{W}$ such that each of its three vertices is a blue neighbor of $u$. Then it is easy to see that neither of $v_1,v_2,v_3$ is a blue neighbor of $u'$, since any of those edges would create two vertex disjoint rainbow cherries on the vertices $v_1,v_2,v_3,u,u',v,v'$.
    Together with the other rainbow cherries in $\mathcal{W}$, they would form $k+1$ vertex-disjoint rainbow cherries, contradicting the maximality of $\mathcal{W}$. Let $p$ denote the number of rainbow cherries in $\mathcal{W}$ with three blue neighbors of $u$, and let $p'$ denote the number of rainbow cherries in $\mathcal{W}$ with three blue neighbors of $u'$, and assume without loss of generality that $p\le p'$. Then the number of blue neighbors of $u$ in $W$ is at most $3p+2(k-p-p')\le 2k\le 2t-2$, a contradiction.

    We have obtained that there are no two red edges as above. Therefore, there are at most two vertices of $U'$ that are incident to red edges with other endpoint not in $W$, and if there are two, say $u_1$ and $u_2$, then $u_1u_2$ must be a red edge and only one of its endpoints, say $u_1$ can have other red neighbors outside $W$. 
    Let $U''$ be obtained from $U'$ by deleting at most one vertex that has a red neighbor outside $W$. If there are two such vertices $u_1,u_2$ as above, we delete $u_1$.
    Then $U''\ge t+3$, and each vertex of $U''$ has at least $t-1$ red neighbors in $W$. 
    Moreover, there are no red edges inside $U''$. 

    Assume now that there are three independent blue edges $uv$, $u'v'$, $u''v''$ 
    inside $U''$. Observe that for a rainbow cherry $v_1v_2v_3$ in $\mathcal{W}$, altogether there are at most three red edges between these vertices and $u,u',u''$. Indeed, it is possible that one of $u,u',u''$ has three red neighbors in the cherry, or that each of $u,u',u''$ has one red neighbor, but if, say, $u$ has at least two red neighbors, then $u',u''$ has no red neighbors there.

    This implies that the number of red edges between $W$ and $u,u',u''$ is at most $3t-3$. 
    This implies that there is a red neighbor $x$ of $u$ or $u'$ or $u''$ outside $W\cup U''$. 
    Without loss of generality, say $x$ is a red neighbor of $u$. Then $xuv$ is a rainbow cherry outside $W$, a contradiction. This implies that there are no three independent blue edges inside $U''$. By the Tutte-Berge theorem \cite{berg}, either 
    there are two blue triangles inside $U''$, or there are three vertices covering all the blue edges inside $U''$. In the first case, for every cherry in $\cW$, at most one vertex is joined to any of the vertices of the triangles by a red edge. This implies that at most $6k$ red edges go from $W$ to the six vertices of the triangles. Those vertices are only incident to these less than $6t$ red edges, thus at least one of these vertices have red degree less than $t$, a contradiction. 
    In the second case, let $U'''$ be obtained from $U''$ by deleting these at most three vertices, then $|U'''|\ge t\ge k+1$ and there are no blue edges inside $U'''$. Let $U_0$ be an arbitrary set of $k+1$ vertices from $U'''$.

    We are going to find $k+1$ vertex-disjoint rainbow cherries with the vertices of $U_0$ as centers. We will define an auxiliary bipartite graph $G_0$. For each $u\in U_0$, we define vertices $u_r$ and $u_b$, these $2k+2$ vertices form part $A$. Part $B$ is $V(G)\setminus U_0$, and $u_r$ is joined to the red neighbors of $u$, $u_b$ is joined to the blue neighbors of $u$. Recall that each $u_r$ has degree at least $t$, since the red neighbors of $u$ are in $V(G)\setminus U_0$. Similarly, each $u_b$ has degree at least $2t$.
    A matching covering $A$ would create $k+1$ vertex-disjoint rainbow cherries in $G$, since the vertices matched to $u_r$ and $u_b$ form a rainbow cherry with $u$.

    If there is no such matching, then there is a set $A'\subseteq A$ violating Hall's condition, i.e., $A'$ has less than $|A'|$ neighbors in $B$. If $A'$ does not contain any $u_b$, then $|A'|\le k+1\le t$, but the neighborhood of $A'$ contains at least $t$ vertices, a contradiction. If $A'$ contains $u_b$ for some $u$, then $|A'|\le 2k+2\le 2t$ and the degree of $u_b$ is at least $2t$, a contradiction again. 

    We found a matching in $G_0$, thus a larger family of rainbow cherries in $G$, a contradiction completing the proof.
\end{proof}

A \textit{partial $r$-expansion} of a graph $F$ is obtained by taking a copy of $F$ and a subset of its edges, and adding $r-2$ distinct new vertices to those edges. We say that an edge $e$ of the shadow graph is \textit{$q$-heavy} in $\cH$ if there are at least $q$ hyperedges of $\cH$ that contain $e$. Otherwise, we call it \textit{$(q-1)$-light}. We will use the following lemma from \cite{tfan}.

\begin{lemma}[\cite{tfan}]\label{greedy}
    Assume that $\cH$ contains a partial $3$-expansion of $F$ and for each edge $e$ of that copy of $F$, if $e$ was not enlarged by a new vertex, then $e$ is $(|E(F)|+|V(F)|-2)$-heavy. Then $\cH$ contains a copy of $F^3$.
\end{lemma}

The \textit{shadow graph} of a hypergraph $\cH$ contains the 2-element subsets of the hyperedges of $\cH$ as edges. For an edge, we call the vertices that extend it to a hyperedge of $\cH$ simply \textit{extending vertices}.
Let us assume that we are given a copy $B$ of $B_s$ in the shadow graph of $\cH$ and for each edge $e$ we are given an integer $w(e)$. We say that $B$ is \textit{nice} if each edge $e$ has at least $w(e)$ extending vertices outside $B$.

\begin{lemma}\label{lemi}
    Let us assume that for each $e$, $w(e)\le w$ and $q$ is sufficiently large with respect to $s$ and to $w$. Assume that we have a copy of $B_q$ in the shadow graph of $\cH$ such that each page edge is $(w(e)+1)$-heavy and the rootlet edge is $w(e)$-heavy. Then there is a nice $B_s$ inside the $B_q$.
\end{lemma}

\begin{proof}
    Let $vv'$ be the rootlet edge of $B_q$. For each page edge, we pick $w(e)$ vertices distinct from $v,v'$ that extend them to a hyperedge of $\cH$. For the rootlet edge $vv'$ we pick $w(vv')$ such vertices.
     
     Now we define an auxiliary graph $G_0$. Its vertices are the page vertices of the $B_q$ that have not been picked for $vv'$. Two vertices $u,u'$ are joined by an edge if and only if either $u$ is picked for a page edge containing $u'$ or $u'$ is picked for a page edge containing $u$. 
     
     By definition, $G_0$ has at most $(w+1)|V(G_0)|$ edges, thus average degree at most $2w+2$. By the Caro-Wei theorem \cite{caro,wei}, $G_0$ contains an independence set of size at least $|V(G_0)|/(2w+3)\ge s$. We pick an independent set of size $s$ to be the set of page vertices of a $B_s$ with rootlet vertices $v,v'$. This $B_s$ is nice since it does not contain any of the $w(e)$ vertices picked for $e$. Indeed, such a vertex $u$ is not a rootlet vertex and it would be joined in $G_0$ to the page vertex that is an endpoint of $e$, contradicting the fact that the page vertices form an independent set in $G_0$.
\end{proof}

\section{Proof of Theorem \ref{thmnew1}}\label{sec3}

A construction for the lower bound is given below.
Let $X$ be a vertex set of order $t$ and $Y$ be a vertex set of order $n-t$.
Then add all triples which contain one vertex from $X$ and two vertices from $Y$ as hyperedges.
This hypergraph is $B_t^3$-free and has $t \binom{n-t}{2}$ hyperedges.

Let us turn to the upper bound. Let $\cH$ be a $B_t^3$-free hypergraph.
The authors together with Zhou \cite{tfan} proved that for any $t>1$, if $n$ is sufficiently large, then
  $\ex_3(n,F_t^3)=\binom{n}{3}-\binom{n-t}{3}=(1+o(1))t\binom{n}{2}$. We start by following the main idea of their proof. 
  We define the following sets of edges and hyperedges.

 Let $E_1$ denote the set of edges that are $t$-light, i.e., are contained in at most $t$ hyperedges, $E_2$ the set of $(t+1)$-heavy, $2t$-light edges, $E_3$ the set of $(2t+1)$-heavy, $3t$-light edges and $E_4$ the set of $(3t+1)$-heavy edges. For $1\le i \le 3$, let $\cH_i$ denote the set of hyperedges that contain at least $i$ edges of $E_i$, let $\cH_4$ denote the set of hyperedges that contain an edge of $E_2$ and two edges of $E_3$, and let $\cH'$ denote the set of other hyperedges in $\cH$. Let $e_i=|E_i|$ and $h_i=|\cH_i|$.


We partition $\cH'$ to several parts. Let $\cH_5$ denote the set of hyperedges in $\cH'$ that contain an edge of $E_2$, an edge of $E_3$ and an edge of $E_4$, $\cH_6$ denote the set of hyperedges in $\cH'$ that contain an edge of $E_2$ and two edges of $E_4$, and $\cH_7$ denote the set of other hyperedges in $\cH'$.



  Next we show that $(2t+1)$-heavy edges do not form large books. This will imply that there are $o(n^2)$ hyperedges where each subedge is $(2t+1)$-heavy.

 \begin{clm}\label{claim3.1}
     $h_7+h_3=o(n^2)$. 
 \end{clm}

 \begin{proof} Let $s$ be sufficiently large.
     Assume that there is a $B_s$ in the shadow graph $G'$ of $\cH_7\cup \cH_3$. Recall that each of the edges is $(2t+1)$-heavy. Let $vv'$ be the rootlet edge. For each page edge $e$ of this $B_s$, we let $w(e)=2t$.
     For the rootlet edge we let $w(vv')=2t+1$.
     
     

We apply Lemma \ref{lemi} to obtain a nice copy $B$ of $B_t$. Then we go through the page edges one by one and choose one of the $2t$ extending vertices outside $B$ to extend it to a hyperedge. We can do this such that we choose distinct vertices, since we have $2t$ choices. Finally, we choose an extending vertex for $vv'$. We have to avoid the $2t$ vertices chosen to extend the page edges, and we have $2t+1$ choices, thus this is also doable.

     We have proved that $G'$ is $B_s$-free. Each hyperedge of $\cH_7\cup \cH_3$ induces a triangle in $G'$, thus $h_7+h_3$ is at most the number of triangles in $G'$. By a theorem of Alon and Shikhelman \cite{alon}, an $n$-vertex $B_s$-free graph contains $o(n^2)$ triangles, completing the proof of the claim. 
 \end{proof}

 \begin{clm}\label{claim3.2}
     $h_1+h_2+h_3+7h_4/6+13h_5/12+h_6\le t(e_1+e_2+e_3+e_4)\le t\binom{n}{2}$.
 \end{clm}

 \begin{proof} We will count the number of pairs $(e,h)$ where $e$ is an edge in $E_i$ and $h$ is a hyperedge in some of the hypergraphs $\cH_j$.
     Each edge in $E_1$ is contained in at most $t$ hyperedges of $\cH_1$, thus for $i=1$ and $h\in \cH_1$, the number of pairs $(e,h)$ is at most $te_1$. On the other hand, each hyperedge in $\cH_1$ contains at least one edge in $E_1$, thus the number of pairs $(e,h)$ is at least $h_1$. This shows that $h_1\le te_1$.

     For $e\in E_2$, we consider each $\cH_j$ that has edges from $E_2$ in its definition, i.e., $h\in \cH_2\cup \cH_4\cup\cH_5\cup \cH_6$. Any edge of $E_2$ is contained in at most $2t$ hyperedges, thus the number of pairs $(e,h)$ is at most $2te_2$. On the other hand, each hyperedge in $\cH_2$ contains at least two edges in $E_2$ and each hyperedge in $\cH_4\cup \cH_5\cup\cH_6$ contains one edge of $E_2$, thus the number of pairs $(e,h)$ is at least $2h_2+h_4+h_5+h_6$. This shows that $h_2+h_4/2+h_5/2+h_6/2\le te_2$.

For $e\in E_3$, we consider each $\cH_j$ that has edges from $E_3$ in its definition except for $\cH_7$, i.e., $h\in \cH_3\cup \cH_4\cup \cH_5$. Any edge of $E_3$ is contained in at most $3t$ hyperedges, thus the number of pairs $(e,h)$ is at most $3te_3$. On the other hand, each hyperedge in $\cH_3$ contains three edges in $E_3$, each hyperedge in $\cH_4$ contains two edges of $E_3$ and each hyperedge in $\cH_5$ contains an edge of $E_3$, thus the number of pairs $(e,h)$ is at least $3h_3+2h_4+h_5$. This shows that $h_3+2h_4/3+h_5/3\le te_3$.

     It is left to show that $h_5/4+h_6/2\le te_4$, i.e., $\frac{3}{2} h_5+3h_6\le 6te_4$. Given an edge $e\in E_4$, let $p(e)$ denote the number of hyperedges of $\cH_5$ containing $e$ and $q(e)$ denote the number of hyperedges of $\cH_6$ containing $e$. Then $\sum_{e\in E_4} p(e)=h_5$ and $\sum_{e\in E_4} q(e)=2h_6$.

We claim that for any $e=vv'\in E_4$ we have $p(e)+q(e)\le 4t$. Assume otherwise. Let $U$ be the set of vertices that extend $e$ to a hyperedge of $\cH_5\cup \cH_6$. We define an auxiliary red-blue graph on $V(\cH)\setminus \{v,v'\}$. Recall that for each $u\in U$, one of $vu$ and $v'u$ is $(2t+1)$-heavy and the other is $(t+1)$-heavy. Let us assume that $v_0u$ is $(2t+1)$-heavy and $v_1u$ is $(t+1)$-heavy, where $\{v,v'\}=\{v_0,v_1\}$. We join $u\in U$ to $w\in V(\cH)\setminus \{v,v'\}$ with a blue edge if $w$ is one of the at least $2t+1$ vertices that extend $v_0u$ to a hyperedge of $\cH$ and by a red edge if $w$ is one of the at least $t+1$ vertices that extend $v_1u$ to a hyperedge of $\cH$. Note that $vv'u$ may be a hyperedge of $\cH$, but all the other hyperedges containing $u$ and one of $v$ and $v$ create edges in this red-blue graph. Therefore, each vertex of $U$ is incident to at least $t$ red edges and at least $2t$ blue edges, hence we can apply Lemma \ref{lem}. Then we find $t$ rainbow cherries, and we use them to construct a $B_t^3$. The rootlet vertices are $v$ and $v'$, and the page vertices are the centers of the cherries. One of the red and blue neighbors of a vertex $u$ extend $vu$ to a hyperedge of $\cH$ and the other extends $v'u$. Finally, the rootlet edge is in $E_4$, thus has at least $3t+1$ extending vertices. We can pick one that has not been used to complete the $B_t^3$ and find a contradiction, proving that $p(e)+q(e)\le 4t$.

Using the above results, we have $h_5+2h_6=\sum_{e\in E_4} p(e)+\sum_{e\in E_4} q(e)\le 4te_4$. This implies that $3h_5/2+3h_6\le 6te_4$. Combining the inequalities we have obtained, $h_1+h_2+h_3+7h_4/6+13h_5/12+h_6\le t(e_1+e_2+e_3+e_4)\le t\binom{n}{2}$.
 \end{proof}

Combining the above claims, we obtain the upper bound

\begin{equation}\label{equ}
   |\cH|= h_1+h_2+7h_4/6+13h_5/12+h_6+h_3+h_7\le t(e_1+e_2+e_3+e_4)+o(n^2)\le t\binom{n}{2}+o(n^2).
\end{equation}
Assume now that $\cH$ has $\ex(n,B_t^3)$ edges. Then in the above calculations, each bound is sharp apart from an additive term $o(n^2)$. In particular, $h_4=o(n^2)$ and $h_5=o(n^2)$. Moreover, we claim that all but $o(n^2)$ edges of $E_2$ are $2t$-heavy. Indeed, otherwise in the proof of Claim \ref{claim3.2} the number of pairs $(e,h)$ with $e\in E_2$, $e\subset h$ is upper bounded by $2te_2-\Omega(n^2)$, which implies that (\ref{equ}) changes to $|\cH|\le t\binom{n}{2}-\Omega(n^2)+o(n^2)$ and we are done. 

\begin{clm}\label{claim3.3}
    $h_6=o(n^2)$.
\end{clm}

\begin{proof}
    Let $\cH_6'$ denote the set of hyperedges in $\cH_6$ where the subedge that belongs to $E_2$ is $2t$-heavy, and $\cH_6''$ the rest of the hyperedges in $\cH_6$. Then $|\cH_6''|$ is at most $2t-1$ times the number of $(2t-1)$-light edges in $E_2$, which is $o(n^2)$.

  Let $G'$ denote the shadow graph of $\cH_6'$. We are going to show that for sufficiently large $s$, there is no $B_s$ in $G'$. Recall that in the proof of Claim \ref{claim3.1}, we showed that there is no $B_s$ formed by $(2t+1)$-heavy edges. We do not improve this by showing that there is no $B_s$ formed by $2t$-heavy edges; we only show that there is no $B_s$ formed by $2t$-heavy edges that are contained by a hyperedge of $\cH_6'$. 
  We cannot apply Lemma \ref{lemi} directly, we need to add some simple modifications, thus we mostly repeat its proof.

  Assume that there is a $B_s$ in $G'$ with rootlet edge $vv'$. Let $u$ be a vertex that extend $vv'$ to a hyperedge of $\cH_6'$. For each page edge of the $B_s$ not incident to $u$, we pick $2t-2$ vertices distinct from $u,v,v'$ that extend them to a hyperedge of $\cH$. One of the edges $vv',uv,uv'$ is in $E_2$, for that edge we pick $2t-1$ such vertices. Then we define the auxiliary graph $G_0$ as in the proof of Lemma \ref{lemi}, with vertex set being the page vertices except for $u$ and for the $2t-1$ vertices picked for the subedge of $uvv'$ in $E_2$. We find an independent set of size $t-1$ as in the proof of Lemma \ref{lemi}. These will be page vertices. We go through the $2t-2$ page edges one by one, selecting one of the $2t-2$ vertices picked for that edge. We can choose distinct vertices this way. Afterwards, we select an extending vertex for the subedge of $uvv'$ in $E_2$. We can pick a new vertex out of the $2t-1$ vertices picked for that edge. Finally, since the remaining two edges of the book are $3t+1$-heavy, we can pick new extending vertices for them by Lemma \ref{greedy}.

  We have shown that $G'$ is $B_s$-free. Each hyperedge of $\cH_6'$ induces a triangle in $G'$, thus $h_6$ is at most the number of triangles in $G'$. By a theorem of Alon and Shikhelman \cite{alon}, an $n$-vertex $B_s$-free graph contains $o(n^2)$ triangles, completing the proof of the claim. 
\end{proof}



\begin{clm}
    $h_2=o(n^2)$.
\end{clm}

\begin{proof}
    The proof is similar to that of Claim \ref{claim3.3}, but we start with a more complicated way of removing a negligible number of hyperedges. First, we claim that all but $o(n^2)$ edges of $E_2$ are in exactly $2t$ hyperedges of $\cH_2$. Indeed, otherwise in the proof of Claim \ref{claim3.2} the number of pairs $(e,h)$ with $e\in E_2$, $e\subset h$ is upper bounded by $2te_2-\Omega(n^2)$, which implies that (\ref{equ}) changes to $|\cH|\le t\binom{n}{2}-\Omega(n^2)+o(n^2)$ and we are done. Observe that it is the exact same argument that showed that all but $o(n^2)$ edges of $E_2$ are $2t$-heavy. However, now we also know that the number of hyperedges outside $\cH_1\cup\cH_2$ is $o(n^2)$, thus the non-negligible contribution of the edges in $E_2$ must come from hyperedges in $\cH_2$. 

    Let us delete the hyperedges from $\cH_2$ that contain an edge of $E_2$ that is not in $2t$ hyperedges of $\cH_2$, this way we deleted $o(n^2)$ hyperedges. Furthermore, we choose a sufficiently large $s$ and delete the hyperedges from $\cH_2$ that do not contain an $s$-heavy edge. We claim that we again have deleted $o(n^2)$ hyperedges. Recall that (\ref{equ}) is simplified to $|\cH_1|+|\cH_2|\le t(e_1+e_2)$, thus we have $e_3+e_4=o(n^2)$. This implies that there are $o(n^2)$ hyperedges that contain an $(s-1)$-light edge from $E_3\cup E_4$. Finally, we have to count the hyperedges in $\cH_2$ that contain three edges of $E_2$. When we count the pairs in the proof of Claim \ref{claim3.2}, $x$ such hyperedges mean that $\cH_2$ contributes $2h_2+x$ to the lower bound. Therefore, (\ref{equ}) becomes $h_1+h_2+x/2\le t\binom{n}{2}+o(n^2)$, we are done unless $x=o(n^2)$.

    Let $\cH_2'$ denote the resulting hypergraph. Now we delete the hyperedges from $\cH_2'$ where the $s$-heavy edge is not contained in at least $s$ hyperedges of $\cH_2'$. There are $o(n^2)$ $s$-heavy edges, and for each of them, we deleted at most $s-1$ hyperedges of $\cH_2'$, thus again we have deleted $o(n^2)$ hyperedges. Let $\cH_2''$ denote the resulting hypergraph.
    
    Consider an arbitrary hyperedge $uvz$ of $\cH_2''$ and let $uv$ be the $s$-heavy edge. Then we have vertices $z_1,\dots, z_{s-1}$ that each extend $uv$ to a hyperedge in $\cH_2'$. We let $w(e)=2t-1$ for the page edges and $w(e)=3t+1$ for the rootlet edge. Then by Lemma \ref{lemi}, we find a nice copy of $B_{3t+1}$, let $S$ be the set of page vertices.
    We try to build a $B_t^3$ with any set of $t$ vertices from $S$ as page vertices and $u,v$ as rootlet vertices. First we try to pick distinct extending vertices to the page edges greedily. If we can, then we can choose an extending vertex for the rootlet edge, since it is $s$-heavy and we are done. If we cannot, then the $2t$ page edges each have the same $2t-1$ extending vertices $x_1,\dots, x_{2t-1}$. This holds for any $t$-subset of $S$, thus for any $i\le 2t-1$, the edges $ux_i$ and $vx_i$ form hyperedges with any vertex of $S$. In particular, the edges $ux_i$ and $vx_i$ are $(3t+1)$-heavy. This means that each edge in the $B_t$ with rootlet vertices $uv$ and page vertices $x_i$ is $(3t+1)$-heavy, thus we find a copy of $B_{2t-1}^3$, a contradiction.

    We obtained that there is no hyperedge in $\cH_2''$, completing the proof.
\end{proof}

By the above claims, we can assume that all but $o(n^2)$ hyperedges are in $\cH_1$. Moreover, (\ref{equ}) turns to $|\cH|\le t|E_1|+o(n^2)$, hence $|E_1|=\binom{n}{2}-o(n^2)$. Let $E_1'$ denote the edges that are in exactly $t$ hyperedges of $\cH_1$ such that the $t$ hyperedges each contain two $5t$-heavy edges. Let $\cH_1''$ denote the hyperedges that contain an edge of $E_1'$ and two $5t$-heavy edges.

\begin{clm}
    $|E_1'|=\binom{n}{2}-o(n^2)$ and $\cH_1''$ contains all but $o(n^2)$ hyperedges of $\cH$.
\end{clm}

\begin{proof} 
Let us first delete from $\cH_1$ the hyperedges that contain at least two edges from $E_1$ to obtain $\cH_1'$. When we count the pairs in the proof of Claim \ref{claim3.2}, $y$ such hyperedges mean that $\cH_1$ contributes $h_1+y$ to the lower bound. Therefore, (\ref{equ}) becomes $h_1+h_2+y\le t\binom{n}{2}+o(n^2)$, we are done unless $x=o(n^2)$.

Let us delete from $\cH_1'$ the hyperedges that contain a $(5t-1)$-light, $(t+1)$-heavy edge. We can count the deleted hyperedges by picking such an edge $o(n^2)$ ways, and hyperedges containing them at most $5t-1$ ways. Then we delete the hyperedges that contain an edge that is in at most $(t-1)$ hyperedges of the current hypergraph. 
    We repeat this as long as we can. We claim that the resulting hypergraph is $\cH_1''$. Indeed, in the resulting hypergraph, each hyperedge contains two $5t$-heavy edges and an edge of $E_1$, but that edge is in at least $t$ hyperedges of the resulting hypergraph.

In the end, we have deleted $x$ edges and at most $(t-1)x$ hyperedges from $\cH_1'$. Then we have at most $\binom{n}{2}-x$ edges remaining (each in $E_1'$), and at most $t(\binom{n}{2}-x)$ hyperedges remaining. Then the total number of hyperedges in $\cH$ is at most $t(\binom{n}{2}-x)+(t-1)x+o(n^2)=t\binom{n}{2}-x+o(n^2)$, thus $x=o(n^2)$ and we are done.
\end{proof}

Consider a triangle $uvw$ in $E_1'$ and let $a_1,\dots a_t$ be the extending vertices of $uv$. 

If $uw$ has extending vertex $b$ and $vw$ has extending vertex $b'$ such that $b,b'\not\in A:=\{a_1,\dots, a_t\}$ and $b\neq b'$, then we have a $B_t^3$ in $\cH$ with rootlet vertices $u,v$ and page vertices $w,a_1,\dots, a_{t-1}$. The edge $uv$ is extended by $a_t$, the edge $uw$ is extended by $b$ and the edge $vw$ is extended by $b'$ to a hyperedge. The rest of the edges are $5t$-heavy, thus we can complete the extending by Lemma \ref{greedy}, a contradiction.

If $b=b'$ and $b, b' \notin A$, then the edges $ub$ and $vb$ are $5t$-heavy. Indeed, they both form a hyperedge with $w$, and $uw$ and $vw$ are in $E_1'$. By the definition of $E_1'$, the other edges of such hyperedges are $5t$-heavy. We have a $B_t^3$ in $\cH$ with rootlet vertices $u,v$ and page vertices $b,a_1,\dots, a_{t-1}$. The edge $uv$ is extended by $a_t$, and the rest of the edges are $5t$-heavy, thus we can complete the extending by Lemma \ref{greedy}, a contradiction. 

We have obtained that at least one of the other edges of the triangle $uvw$, say $uw$ has the same set $A=\{a_1,\dots, a_t\}$ of extending vertices. But the same holds if we start with the edge $vw$, thus the extending vertices of $vw$ are also $A$.

Let $U$ denote the set of vertices that are incident to $n-o(n)$ edges of $E_1'$, then clearly $|U|=n-o(n)$. Then there is an edge $uv$ inside $U$, and there are $n-o(n)$ common neighbors of $u$ and $v$ in $E_1'$. This means that all of the edges joining $u$ or $v$ to these common neighbors have the same set $A$ as their extending vertices. In particular, all of those vertices have that they are joined to each vertex of $A$ by a $5t$-heavy edge.

Let $U'$ denote the set of vertices $u$ such that $ua_i$ is $5t$-heavy for each $i$. By the above, $|U'|=n-o(n)$. Let $W=V(\cH)\setminus (A\cup U')$ and $m:=|W\cup A|$. Consider now $u,u'\in U'$. We claim that all the extending vertices of $uu'$ are in $A$ (note that such an edge is not necessarily in $E_1'$). Indeed, if $b\not\in A$ is an extending vertex, then we have a $B_t^3$ in $\cH$ with rootlet vertices $u,v$ and page vertices from $A$. The extending vertex of $uu'$ is $b$, and the rest of the edges are $5t$-heavy, thus we can complete the extending by Lemma \ref{greedy}, a contradiction.

Let us consider now the hyperedges of $\cH$ according to the location of their vertices in the partition to $U',W,A$. There are at most $t\binom{n-m}{2}$ hyperedges with at least 2 vertices in $U'$. There are at most $t\binom{m}{2}+o(m^2)=o(nm)$ hyperedges inside $A\cup W$ by the asymptotic result (\ref{equ}). It is left to consider the hyperedges with one vertex inside $U'$. 

Assume that there is a hyperedge with two vertices in $A$, say $a_1$ and $a_2$ and let $v$ be its third vertex. Recall that $a_1$ and $a_2$ are joined by $5t$-heavy edges to each vertex of $U'$. We pick $t$ arbitrary vertices from $U'$ different from $v$ as page vertices and $a_1,a_2$ as rootlet vertices of a $B_t$. Then the rootlet edge is extended by $v$, and the other edges are $5t$-heavy, thus we can complete the extending by Lemma \ref{greedy}, a contradiction.

For any vertex $w\in W$, there is a vertex $a\in A$ such that $wa$ is $(5t-1)$-light. The other vertices of $A$ are in at most $n$ hyperedges together with $w$, thus there are at most $(t-1)(m-t)n+(5t-1)m$ hyperedges with one vertex in each of $U',W,A$.

Finally, let $\cG$ denote the hyperedges $uw_1w_2$ with $u\in U'$, $w_1,w_2\in W$ and let $\cG'$ denote the hyperedges in $\cG$ such that both $uw_1$ and $uw_2$ are in exactly one hyperedge of $\cG$. Then $|\cG'|\le n(m-t)/2$, since they contain two edges between $U'$ and $W$, and each edge is counted at most once. Let $\cG''$ denote the rest of the hyperedges of $\cG$.

Let $w\in W$ and let $U(w)$ denote the set of vertices $u\in U'$ that have at least two hyperedges that contain $u,w$ and a third vertex from $W$. If $u_1,u_2\in U(w)$, then we can pick distinct vertices $w_1,w_2\in W$ such that $u_1w_1w$ and $u_2w_2w$ are both hyperedges in $\cH$. If there is any hyperedge containing $u_1u_2$, then the third vertex of that hyperedge is in $A$, say $a_t$. Then we find a $B_t^3$ in $\cH$ with rootlet vertices $u_1,u_2$ and page vertices $w,a_1,\dots, a_{t-1}$. The extending vertex for $u_1u_2$ is $a_t$, and $w_i$ for $u_iw$. The rest of the edges are $5t$-heavy, thus we can complete the extending by Lemma \ref{greedy}, a contradiction.

This implies that $U(w)$ forms an independent set in $E_1'$, thus $|U(w)|=o(n)$. 
Let $u\in U(w)$. If $uw$ is $3t$-light, then we delete the hyperedges containing $uw$ from $\cG''$, this way we delete $o(nm)$ hyperedges. Let $\cG'''$ denote the resulting hypergraph. Let $w_1,w_2\in W$ and assume that there are $3t+1$ hyperedges in $\cG'''$ containing them. Then for any $t$ of those hyperedges, their shadow forms a $B_t$ where each edge is $(3t+1)$-heavy. By Lemma \ref{greedy}, we find a $B_t^3$ in $\cG'''$, a contradiction. Therefore, there are at most $3t\binom{m}{2}$ hyperedges in $\cG'''$.

Now we are ready to count the total number of hyperedges. By adding up the bounds obtained above, the number of hyperedges is at most

\begin{equation*}
    t\binom{n-m}{2}+o(nm)+(t-1)(m-t)n+(5t-1)m+n(m-t)/2+o(nm)+3t\binom{m}{2}.
\end{equation*}

If $m>t$ and $n$ is sufficiently large, then this is less than $t\binom{n-t}{2}$, completing the proof.

\section{Concluding remarks}\label{sec4}

A natural question is what happens for larger uniformity. We cannot answer this question, but we can determine the order of magnitude. More generally, let $B_{t,k,s}$ denote the graph that consists of $t$ copies of $K_k$ sharing the same set of $s$ vertices. Note that the proposition below generalizes Proposition 4.1 in \cite{tfan}.

\begin{proposition} Let $k>2$ and $t>1$. Then
    \begin{displaymath}
\ex_r(n,B_{t,k,s}^r)=
\left\{ \begin{array}{l l}
\Theta(n^{r-1}) & \textrm{if\/ $k\le r$},\\
(1+o(1))\binom{k-1}{r}\left(\frac{n}{k-1}\right)^r & \textrm{if\/ $k>r$}.\\
\end{array}
\right.
\end{displaymath}
\end{proposition}

We will use a connection to \textit{generalized Tur\'an numbers}. Let $\ex(n,H,F)$ denote the largest number of copies of $H$ in $n$-vertex $F$-free graphs. The systematic study of this topic was initiated by Alon and Shikhelman \cite{alon}, see \cite{GePa} for a survey. Gerbner \cite{Ge1} proved that for any $F$ and $r$,  we have $\ex_r(n,F^r)=O(n^{r-1})+\ex(n,K_r,F)$.

\begin{proof}[\bf Proof] The case $k>r$ is a theorem in \cite{MuVe}, see \cite{pttw} for a proof (it also easily follows from the above mentioned result of Gerbner combined with a theorem of Alon and Shikhelman \cite{alon}).
In the case $k\le r$, the lower bound is given by taking all the hyperedges containing a fixed element. Indeed, this hypergraph does not contain two independent hyperedges, while $B_{t,k,s}$ does.

    Let us continue with the upper bound in the case $k\le r$. By the theorem of Gerbner mentioned above, it is enough to prove that $\ex(n,K_r,B_{t,k,s})=O(n^{r-1})$. There are several papers proving bounds on $\ex(n,K_r,B_{2,k,s})$, see e.g. \cite{gp,zz}, but we are not aware of any results for general $t$, except for the cases $s\le 2$, see \cite{zcggyh,gerb,ger}. However, it is easy to prove the bound $O(n^{r-1})$.
    
    Let $G$ be an $n$-vertex $B_{t,k,s}$-free graph and let $\cH$ denote the hypergraph on the same vertex set that has the $r$-cliques of $G$ as hyperedges. It is easy to see that $t$ hyperedges of $\cH$ cannot share exactly $s$ vertices. In other words, the $t$-wise intersections can have sizes $0,1,\dots, s-1, s+1,\dots, r-1$. That is $r-1$ different sizes, thus by a theorem of Grolmusz and Sudakov \cite{gs}, $|\cH|\le (t-1)\sum_{i=1}^{r-1}\binom{n}{i}=O(n^{r-1})$. This implies that $\ex(n,K_r,B_{t,k,s})=O(n^{r-1})$, completing the proof.
\end{proof}


\bigskip
\textbf{Funding}: 
The research of Cheng is supported by the National Natural Science Foundation of China (Nos. 12131013 and 12471334), Shaanxi Fundamental Science Research Project for Mathematics and Physics (No. 22JSZ009) and the China Scholarship Council (No. 202406290241). 

The research of Gerbner is supported by the J\'anos Bolyai research scholarship of the Hungarian Academy of Sciences.

\bigskip
\textbf{Statement of AI use}. We tried to use ChatGPT 5.4 to obtain a proof of Lemma 2.1. We gave it the statement and it gave us multiple incorrect proofs, thus we ended up proving the lemma ourselves. We could use from those failed proofs the basic idea of defining $\cW$ and then using the vertices of $U'$ as centers.

\end{document}